\documentclass[oneside,english]{amsart}
\usepackage[T1]{fontenc}
\usepackage[utf8]{inputenc}
\usepackage{float}
\usepackage{textcomp}
\usepackage{url}
\usepackage{amsbsy}
\usepackage{amstext}
\usepackage{amsthm}
\usepackage{amssymb}

\makeatletter

\floatstyle{ruled}
\newfloat{algorithm}{tbp}{loa}
\providecommand{\algorithmname}{Algorithm}
\floatname{algorithm}{\protect\algorithmname}

\numberwithin{equation}{section}
\numberwithin{figure}{section}
\theoremstyle{plain}
\newtheorem{thm}{\protect\theoremname}
\theoremstyle{plain}
\newtheorem{prop}{\protect\propositionname}
\theoremstyle{plain}
\newtheorem{lem}{\protect\lemmaname}
\theoremstyle{definition}
\newtheorem{claim}{\protect\claimname}

\makeatletter
\@namedef{subjclassname@2020}{%
  \textup{2020} Mathematics Subject Classification}
\makeatother
\subjclass{11M26, 11N05, 42A75}

\usepackage{pseudo}

\makeatother

\usepackage{babel}
\providecommand{\claimname}{Claim}
\providecommand{\lemmaname}{Lemma}
\providecommand{\propositionname}{Proposition}
\providecommand{\theoremname}{Theorem}

\begin{document}
\title{On the sign changes of $\psi(x)-x$}
\author{M.~Grześkowiak, J.~Kaczorowski, Ł.~Pańkowski, M.~Radziejewski}
\keywords{Prime number theorem; oscillations; Riemann zeta-function; zeros of
analytic almost periodic functions.}
\dedicatory{\raggedleft For Alberto Perelli\\
in celebration of his 70th birthday}
\begin{abstract}
We improve the lower bound for $V(T)$, the number of sign changes
of the error term $\psi(x)-x$ in the Prime Number Theorem in the
interval $[1,T]$ for large $T$. We show that
\[
\liminf_{T\to\infty}\frac{V(T)}{\log T}\geq\frac{\gamma_{0}}{\pi}+\frac{1}{60}
\]
where $\gamma_{0}=14.13\ldots$ is the imaginary part of the lowest-lying
non-trivial zero of the Riemann zeta-function. The result is based
on a new density estimate for zeros of the associated $k$-function,
over $4\cdot10^{21}$ times better than previously known estimates
of this type.
\end{abstract}

\begingroup
\def\uppercasenonmath#1{} 
\let\MakeUppercase\relax 
\maketitle
\endgroup

\section{Introduction}

Let $V(T)$ denote the number of sign changes of $\psi(x)-x$ in the
interval $[1,T]$, where $\psi(x)=\sum_{n\leq x}\Lambda(n)$ denotes
the familiar prime counting function. The main objective of this paper
is to prove the following result. 
\begin{thm}
\label{thm:Main-Theorem}We have 
\[
\liminf_{T\to\infty}\frac{V(T)}{\log T}\geq\frac{\gamma_{0}}{\pi}+\frac{1}{60},
\]
 where $\gamma_{0}=14.13\ldots$ denotes the ordinate of the lowest-lying
non-trivial zero of the Riemann zeta-function. 
\end{thm}
The problem of estimating $V(T)$ from below has a long history. A.~E.~Ingham~\cite{Ingham}
proved in 1936, assuming a kind of quasi-Riemann Hypothesis, that
the limit $\kappa_{0}:=\liminf_{T\to\infty}\frac{V(T)}{\log T}$ was
positive. Nearly 50 years later, the second named author proved the
same result unconditionally with $\kappa_{0}\geq\gamma_{0}/(4\pi)$,
cf.~\cite{Kaczorowski1984}. Soon after, he improved this to $\kappa_{0}\geq\gamma_{0}/\pi$,
cf.~\cite{Kaczorowski1985}. The same constant was previously obtained
for the $\limsup$ instead of $\liminf$ by G. P\'{o}lya~\cite{Polya1969}.
Breaking the $(\gamma_{0}/\pi)$-barrier required an essentially new
idea and was achieved in~\cite{Kaczorowski1991}. However, the improvement
was not large: $\kappa_{0}\geq\gamma_{0}/\pi+2\cdot10^{-248}$. A
better result, $\kappa_{0}\geq\gamma_{0}/\pi+2\cdot10^{-25}$, was
announced in~\cite{Kaczorowski1996}, but no proof of this result
was ever published. Recently, T.~Morrill, D.~Platt, and T.~Trudgian~\cite{MorrillPlattTrudgian2022}
published a full proof of the estimate $\kappa_{0}\geq\gamma_{0}/\pi+1.867\cdot10^{-30}$.
Theorem~\ref{thm:Main-Theorem} supersedes this result. First we
outline (below) the ideas leading to Theorem~\ref{thm:Main-Theorem}
in the most difficult case, when the Riemann Hypothesis is true. The
complete proof is contained in the following sections.

It was shown by the second author~\cite[Corollary 1.1]{Kaczorowski1991}
that $\kappa_{0}$ is related to the density of zeros of the normalized
$k$-function
\[
F^{*}(z)=1+\sum_{n=1}^{\infty}\frac{\rho_{0}}{\rho_{n}}e^{i(\gamma_{n}-\gamma_{0})z},\qquad z\in\mathbb{H},
\]
where $\mathbb{H}$ is the upper half-plane $\left\{ z\in\mathbb{C}:\Im z>0\right\} $
and $\rho_{n}=\frac{1}{2}+i\gamma_{n}$ denotes the $n$th non-trivial
zero of the Riemann zeta function in the upper half-plane, starting
with $\rho_{0}=\frac{1}{2}+i14.134725\dots$, counted according to
their multiplicity, so $\gamma_{0}\leq\gamma_{1}\leq\gamma_{2}\leq\dotsc$
\begin{thm}[{cf.~\cite[Corollary 1.1]{Kaczorowski1991}}]
\label{thm:JK-Corollary-1.1}Suppose the Riemann Hypothesis is true.
Let
\[
\varkappa=\lim_{Y\to0^{+}}\lim_{T\to\infty}\frac{1}{T}\#\left\{ \xi=x+iy:F^*(\xi)=0,0<x<T,y\geq Y\right\} .
\]
Then
\[
\liminf_{T\to\infty}\frac{V(T)}{\log T}\geq\frac{\gamma_{0}}{\pi}+2\varkappa.
\]
\end{thm}
In other words, we have $\kappa_{0}\geq\frac{\gamma_{0}}{\pi}+2\varkappa$,
and so, in order to prove Theorem~\ref{thm:Main-Theorem} we need
to show a lower bound for $\varkappa$. The proof will rest on a general
theorem on zeros of holomorphic almost periodic functions on the upper
half-plane, which is of an independent interest. Suppose $(a(n))_{n=1}^{\infty}$
and $(\omega_{n})_{n=1}^{\infty}$ are two sequences of numbers such
that $a(n)\in\mathbb{C}$, $\omega_{n}\in\mathbb{R}$, $0<\omega_{1}<\omega_{2}<\ldots$
and 
\[
\sum_{n=1}^{\infty}\left|a(n)\right|e^{-\omega_{n}y}<\infty\qquad\text{for all \ensuremath{y>0}}.
\]
For $z\in{\mathbb{H}}$ we define 
\begin{equation}
F(z)=1+\sum_{n=1}^{\infty}a(n)e^{i\omega_{n}z}.\label{eq:F}
\end{equation}
This is a holomorphic, almost periodic function on $\mathbb{H}$.
For $N\geq1$ and $z=x+iy\in{\mathbb{H}}$ we set 
\begin{equation}
F_{N}(z):=1+\sum_{n=1}^{N}a(n)e^{i\omega_{n}z},\label{eq:F_N}
\end{equation}
 
\begin{equation}
R_{N}(y):=\sum_{N+1}^{\infty}|a(n)|e^{-\omega_{n}y}.\label{eq:R_N}
\end{equation}
Let $\left\Vert x\right\Vert $ denote the norm of a real number $x$,
i.e. its distance from the nearest integer. The basic idea, introduced
in~\cite[pp. 54–55]{Kaczorowski1991}, is to find a zero $\xi_{0}$
of $F(x+iy)$ and then use the almost-periodicity of $F(x+iy)$ in
$x$ to show that there is a sequence of ``copies'' of $\xi_{0}$.
We can actually use a zero of the partial sum $F_{N}(z)$ for $N=8$,
say, and encircle it with a contour $z(t)$, $0\leq t\leq1$, such
that the values $F_{N}(z(t))$ go around $0$ by a distance sufficiently
greater than the maximal approximation error $R_{N}(\Im z)$. For
appropriate shifts $\tau$, i.e. such that all the norms $\left\Vert \frac{\omega_{n}\tau}{2\pi}\right\Vert $,
for $n=1,\dotsc,N$, are sufficiently small, the shifted contour $z(t)+\tau$
must contain a zero of $F$, because each summand $a(n)e^{i\omega_{n}(z+\tau)}$
will be close to $a(n)e^{i\omega_{n}z}$ so that the argument of $F(z(t)+\tau)$
still grows by $2\pi$ as $t$ goes from $0$ to $1$, so there must
be a zero inside the contour $z(t)+\tau$ by the argument principle.
The problem reduces to counting the density of ``appropriate'' shifts
with the additional constraint that the distinct $\tau$'s should
differ by at least $\delta$, the width of our contour, so that no
zero will be counted twice. The improvement obtained here involves
two new ideas:
\begin{itemize}
\item [$(i)$] a more flexible (although slightly technical) sufficient
condition for the shifted contour to contain a zero, where the uniform
bound on the norms $\left\Vert \frac{\omega_{n}\tau}{2\pi}\right\Vert $
is replaced by a weighted sum, cf.~Proposition~\ref{prop:sufficient-condition},
\item [$(ii)$] a new method, called the \emph{tiling method}, for estimating
the density of appropriate shifts, cf.~Theorem~\ref{thm:Tiling-method}.
\end{itemize}
Proposition~\ref{prop:sufficient-condition} also involves shifting
the phases of the summands of $F_{N}$ by arbitrary values $\eta_{1},\dotsc,\eta_{N}\in\mathbb{R}$,
and allows us to alter the expected rate of variation of the function's
argument along the contour. Both of these features are optional, but
contribute to a higher final result.
\begin{prop}
\label{prop:sufficient-condition}Let $\boldsymbol{\eta}=(\eta_{1},\dotsc,\eta_{N})\in\mathbb{R}^{N}$
and denote
\[
F_{N}^{\boldsymbol{\eta}}(z)=1+\sum_{n=1}^{N}a(n)e^{i\left(\eta_{n}+\omega_{n}z\right)}.
\]
Suppose $z_{1}(t)$, $0\leq t\leq1$, is a closed contour in $\mathbb{H}$,
oriented counterclockwise, not passing through any zero of $F_{N}^{\boldsymbol{\eta}}(z)$.
Suppose $z_{2}(t)$ is a closed contour around $0$, also oriented
counterclockwise, satisfying $\left|z_{2}(t)\right|=1$, $0\leq t\leq1$,
and $\eta:[0,1]\to[-\tfrac{\pi}{2},\tfrac{\pi}{2}]$ is an arbitrary
function. Let
\[
\varphi_{n}(t)=\eta_{n}+\arg a(n)+\omega_{n}\Re z_{1}(t)-\arg z_{2}(t)-\eta(t),\qquad0\leq t\leq1,n=1,\dotsc,N.
\]
Suppose $u(t)$ is a function satisfying
\[
0<u(t)<\Re\left(F_{N}^{\boldsymbol{\eta}}(z_{1}(t))\overline{z_{2}(t)}e^{-i\eta(t)}\right)-R_{N}(\Im z_{1}(t)),\qquad0\leq t\leq1.
\]
For $\varphi,\psi\in[0,\pi]$ let 
\[
v(\varphi,\psi)=\begin{cases}
\cos\left(\varphi\right)-\cos\left(\varphi+\psi\right), & \varphi+\psi\leq\pi,\\
\cos\left(\varphi\right)+1, & \text{otherwise.}
\end{cases}
\]
Suppose $w_{n}:[0,\tfrac{1}{2}]\to[0,1]$ are non-decreasing functions
such that
\begin{equation}
w_{n}(x)\geq\sup_{0\leq t\leq1}\frac{\left|a(n)\right|}{u(t)e^{\omega_{n}\Im z_{1}(t)}}v\left(2\pi\left\Vert \frac{\varphi_{n}(t)}{2\pi}\right\Vert ,2\pi x\right)\label{eq:wn(x)-assumption}
\end{equation}
for every $x\in[0,\tfrac{1}{2}]$ and $n=1,\dotsc,N$. Then for every
$\tau\in\mathbb{R}$ such that
\begin{equation}
\sum_{n=1}^{N}w_{n}\left(\left\Vert \frac{\omega_{n}\tau-\eta_{n}}{2\pi}\right\Vert \right)\leq1\label{eq:prop-sum-of-penalties}
\end{equation}
there is a zero of $F(z)$ inside the contour $z_{1}(t)+\tau$.
\end{prop}
Condition~(\ref{eq:prop-sum-of-penalties}) defines a subset of the
$N$-dimensional cube 
\begin{equation}
A=\left\{ \left(x_{n}\right)_{n=1}^{N}\in\left[-\tfrac{1}{2},\tfrac{1}{2}\right]^{N}:\sum_{n=1}^{N}w_{n}\left(\left|x_{n}\right|\right)\leq1\right\} \label{eq:solution-set}
\end{equation}
such that a shift $\tau$ satisfies~(\ref{eq:prop-sum-of-penalties})
whenever $\left(\left\Vert \frac{\omega_{n}\tau-\eta_{n}}{2\pi}\right\Vert \right)_{n=1}^{N}\in A$.
The classical tool to count such $\tau$'s comes from the theory of
Diophantine approximation. Using Minkowski convex body theorem (and,
in addition, assuming all the $\eta_{n}$ to be zero and replacing
$w_{n}$ by convex functions) we would get, cf.~\cite{Montgomery1977},
that the density of such $\tau$'s is at least $2^{-N}\operatorname{vol}(A)$,
leading to
\[
\varkappa\geq2^{-N}\delta^{-1}\operatorname{vol}(A),
\]
as we need count $\tau$'s spaced by at least $\delta$ for non-overlapping
contour shifts. The estimate $2^{-N}\operatorname{vol}(A)$ for the
density of $\tau$'s is quite unsatisfactory if we consider that under
the conjecture of linear independence of $(\gamma_{n})$ over the
rationals~\cite{Wintner1935} this density would be $\operatorname{vol}(A)$,
by the Kronecker-Weyl theorem~\cite[Theorem 1 on p. 359, Appendix, §8]{Karatsuba1992}.
Of course, one can relax the hypothesis that the $\gamma_{n}$ are
linearly independent by using a quantitative form of the Kronecker
approximation theorem (see, for example, Theorem 4.1 in~\cite{GonekMontgomery2016}
or Theorem 1 in~\cite{Chen2000}, whereby we only need to suppose
that all linear combinations of $\gamma_{n}$, $1\leq n\leq N$, with
integer coefficients lying in a fixed bounded interval $[-M,M]$ are
not equal to 0. However, in this type of results the density of $\tau$'s
essentially depends on a lower bound for the absolute value of such
combinations, which is, even for small $N$ and $M$, extremely small.
In consequence, previous results allow us to improve the lower bound
announced in~\cite{Kaczorowski1996}, but do not lead to any result
close to Theorem 1. In Theorem~\ref{thm:Tiling-method} we replace
the estimate $2^{-N}\operatorname{vol}(A)$ with the volume $\operatorname{vol}(A')$
of a subset $A'\subset A$ and we are able to make $\operatorname{vol}(A')$
very close to the conjectural value $\operatorname{vol}(A)$ in our
case. The method involves filling the hypercube $[-\frac{1}{2},\frac{1}{2}]^{N}$
with a dense lattice of reachable points, i.e.~points lying in cosets
of the form~$\left(\frac{\omega_{n}\tau-\eta_{n}}{2\pi}\right)_{n=1}^{N}+\mathbb{Z}^{N}$,
$\tau\in\mathbb{R}$. In essence we show that $\left(\frac{\omega_{n}\tau-\eta_{n}}{2\pi}\right)_{n=1}^{N}$
is equidistributed mod $1$ down to a certain scale, determined by
the `granularity' of our lattice.
\begin{thm}
\label{thm:Tiling-method}Let $\boldsymbol{\eta},\boldsymbol{\theta}\in\mathbb{R}^{N}$,
$\boldsymbol{\theta}=(\theta_{1},\dotsc,\theta_{N})$, and let $A\subseteq\left[-\tfrac{1}{2},\tfrac{1}{2}\right]^{N}$
be a compact set. Let $(\boldsymbol{v}_{k})_{k=1}^{N}$ be a basis
of $\mathbb{R}^{N}$ such that there exist real numbers $t_{1},\dotsc,t_{N}$
satisfying $\boldsymbol{v}_{k}\in t_{k}\boldsymbol{\theta}+\mathbb{Z}^{N}$,
$k=1,\dotsc,N$. Let $d_{1},\dotsc,d_{N}$ be positive numbers such
that
\[
\pm\boldsymbol{v}_{1}\pm\dotsc\pm\boldsymbol{v}_{N}\in D:=\prod_{n=1}^{N}[-d_{n},d_{n}],
\]
for each possible combination of $\pm$ symbols. Let
\[
A'=\left\{ \boldsymbol{v}\in A:\boldsymbol{v}+\frac{3}{2}D\subseteq A\right\} .
\]
Then for every $\delta>0$ there exists an increasing sequence of
non-negative numbers $(\tau_{\ell})_{\ell=1}^{\infty}$ such that
for all $\ell$ we have $\tau_{\ell}\boldsymbol{\theta}\in\boldsymbol{\eta}+A+\mathbb{Z}^{N}$
and $\tau_{\ell+1}\geq\tau_{\ell}+\delta$, and 
\[
\liminf_{\ell\to\infty}\frac{\ell}{\tau_{\ell}}\geq\frac{1}{\delta}\operatorname{vol}(A').
\]
\end{thm}
This results in $\varkappa\geq\delta^{-1}\operatorname{vol}(A')$.
In addition, $A$ no longer needs to be convex, so we can take the
full solution set (\ref{eq:solution-set}) instead of its convex subset.
Underneath we still use the Minkowski convex body theorem (or its
proof), except we apply it to a tiny parallelotope, derived from the
fundamental parallelotope of our lattice:
\begin{equation}
P=\left\{ \sum_{k=1}^{N}t_{k}\boldsymbol{v}_{k}:t_{k}\in\left[-\tfrac{1}{2},\tfrac{1}{2}\right],k=1,\dotsc,N\right\} ,\label{eq:primitive-cell}
\end{equation}
where $\boldsymbol{v}_{k}$ are lattice generators. Whereas in the
classical argument~\cite[proof of Theorem II in Chapter III]{Cassels1971}
we would find a set of $\tau$'s such that $\left(\theta_{n}\tau\right)_{n=1}^{N}$
lie (mod $1$) in a shifted copy of $\frac{1}{2}A$, so their differences
fall into $\frac{1}{2}A-\frac{1}{2}A=A$, here the differences are
in a copy of $P$ shifted by an element $-v(\tau_{0})\in P$. We additionally
shift this copy by lattice elements contained sufficiently deep inside
$A$, as if affixing tiles. The precise position of tiles is not known
(it depends on $v(\tau_{0})$), but neighbouring tiles are aligned
perfectly next to each other, and so almost fill the solution set.
The losses occur near the border and tend to $0$ as the lattice gets
dense, provided the solution set is regular enough. (It does not need
to be convex). The problem now comes down to generating a dense lattice.
We solve it using a dedicated version of the LLL algorithm, as described
in Section~\ref{sec:Diophantine}. 

A side effect of the tiling method is that we no longer need to find
an actual zero of $F_{N}$. Instead we can freely alter the phase
of each summand of $F_{N}$ and consider a contour around the highest-lying
zero $\xi_{N}^{*}$ of
\[
G_{N}(z)=1-\sum_{n=1}^{N}\left|\frac{\rho_{0}}{\rho_{n}}\right|e^{i(\gamma_{n}-\gamma_{0})z}.
\]
The high-lying zeros seem to work best for us and the imaginary part
of $\xi_{N}^{*}$ is greater than that of any zero of $F_{N}(z)$.
The particular choice of the contours $z_{1}(t)$, $z_{2}(t)$ and
the function $\eta(t)$ for applying Proposition~\ref{prop:sufficient-condition}
was based on the authors' experience with a series of unpublished
numerical experiments, of which the final part we make available~\cite{GKPR}.
Here we only note that $z_{2}(t)=F_{N}^{\boldsymbol{\eta}}(z_{1}(t))/\left|F_{N}^{\boldsymbol{\eta}}(z_{1}(t))\right|$
and $\eta(t)=0$ are natural candidates, for which we can take
\[
u(t)=\left|F_{N}^{\boldsymbol{\eta}}(z_{1}(t))\right|-R_{N}(\Im z_{1}(t))-\varepsilon
\]
for some small $\varepsilon>0$. These choices can be tweaked further,
cf. Section~\ref{sec:Proof}, although it would be hard to search
the space of all possible options to find optimal ones.

\textbf{Acknowledgement. }This research was partially supported by
grant 2021/41/B/ST1/00241 from the National Science Centre, Poland. 

\section{A sufficient condition for zeros}

In this section we prove Proposition~\ref{prop:sufficient-condition}
and provide auxiliary estimates needed for its application. Lemma~\ref{lem:RN-bound}
gives an upper bound for $R_{N}(\Im z_{1}(t))$ in the case $F=F^{*}$,
when the Riemann Hypothesis is true and $N$ is large. Lemma~\ref{lem:v-extremes}
allows us to replace the $\sup$ in~(\ref{eq:wn(x)-assumption})
with $\max$ over a finite set. The latter is essential for checking~(\ref{eq:wn(x)-assumption})
numerically.
\begin{proof}[Proof of Proposition~\ref{prop:sufficient-condition}]
Fix $\tau\in\mathbb{R}$ that satisfies (\ref{eq:prop-sum-of-penalties}).
By the argument principle it is enough to show that 
\[
F(z_{1}(t)+\tau)\neq-z_{2}(t)\left|F(z_{1}(t)+\tau)\right|,\qquad0\leq t\leq1,
\]
which follows from
\[
\Re\left(F(z_{1}(t)+\tau)\overline{z_{2}(t)}e^{-i\eta(t)}\right)>0,\qquad0\leq t\leq1.
\]
We have
\begin{align*}
\Re\left(F(z_{1}(t)+\tau)\overline{z_{2}(t)}e^{-i\eta(t)}\right) & \geq\Re\left(F_{N}(z_{1}(t)+\tau)\overline{z_{2}(t)}e^{-i\eta(t)}\right)-R_{N}(\Im z_{1}(t))\\
 & >\Re\left(\Bigl(F_{N}(z_{1}(t)+\tau)-F_{N}^{\boldsymbol{\eta}}(z_{1}(t))\Bigr)\overline{z_{2}(t)}e^{-i\eta(t)}\right)+u(t),
\end{align*}
so it is enough to show
\begin{equation}
\Re\left(\Bigl(-F_{N}(z_{1}(t)+\tau)+F_{N}^{\boldsymbol{\eta}}(z_{1}(t))\Bigr)\overline{z_{2}(t)}e^{-i\eta(t)}\right)\leq u(t),\qquad0\leq t\leq1.\label{eq:sufficient-condition-u(t)}
\end{equation}
We have

\begin{align*}
\Re\biggl(\Bigl(-F_{N}(z_{1}(t)+\tau) & +F_{N}^{\boldsymbol{\eta}}(z_{1}(t))\Bigr)\overline{z_{2}(t)}e^{-i\eta(t)}\biggr)\\
= & \Re\sum_{n=1}^{N}a(n)e^{i\left(\eta_{n}+\omega_{n}z_{1}(t)-\arg z_{2}(t)-\eta(t)\right)}\left(1-e^{i\left(\omega_{n}\tau-\eta_{n}\right)}\right)\\
= & \sum_{n=1}^{N}\left|a(n)\right|e^{-\omega_{n}\Im z_{1}(t)}\left(\cos\left(\varphi_{n}(t)\right)-\cos\left(\omega_{n}\tau-\eta_{n}+\varphi_{n}(t)\right)\right).
\end{align*}
Since 
\[
\cos\left(\omega_{n}\tau-\eta_{n}+\varphi_{n}(t)\right)=\cos\left(2\pi\left(\left\Vert \frac{\omega_{n}\tau-\eta_{n}+\varphi_{n}(t)}{2\pi}\right\Vert \right)\right)
\]
and
\[
\left\Vert \frac{\omega_{n}\tau-\eta_{n}+\varphi_{n}(t)}{2\pi}\right\Vert \leq\left\Vert \frac{\omega_{n}\tau-\eta_{n}}{2\pi}\right\Vert +\left\Vert \frac{\varphi_{n}(t)}{2\pi}\right\Vert ,
\]
moreover $\cos(x)$ is monotone decreasing on $[0,\pi]$, we have
\[
\cos\left(\omega_{n}\tau-\eta_{n}+\varphi_{n}(t)\right)\geq\cos\left(2\pi\left(\left\Vert \frac{\omega_{n}\tau-\eta_{n}}{2\pi}\right\Vert +\left\Vert \frac{\varphi_{n}(t)}{2\pi}\right\Vert \right)\right)
\]
when $\left\Vert \frac{\omega_{n}\tau-\eta_{n}}{2\pi}\right\Vert +\left\Vert \frac{\varphi_{n}(t)}{2\pi}\right\Vert \leq\frac{1}{2}$
and $\cos\left(\omega_{n}\tau-\eta_{n}+\varphi_{n}(t)\right)\geq-1$
in any case. Hence
\[
\cos\left(\varphi_{n}(t)\right)-\cos\left(\omega_{n}\tau-\eta_{n}+\varphi_{n}(t)\right)\leq v\left(2\pi\left\Vert \frac{\varphi_{n}(t)}{2\pi}\right\Vert ,2\pi\left\Vert \frac{\omega_{n}\tau-\eta_{n}}{2\pi}\right\Vert \right)
\]
and
\begin{align*}
\Re\biggl(\Bigl(-F_{N}(z_{1}(t)+\tau) & +F_{N}^{\boldsymbol{\eta}}(z_{1}(t))\Bigr)\overline{z_{2}(t)}\biggr)\\
\leq & \sum_{n=1}^{N}\left|a(n)\right|e^{-\omega_{n}\Im z_{1}(t)}v\left(2\pi\left\Vert \frac{\varphi_{n}(t)}{2\pi}\right\Vert ,2\pi\left\Vert \frac{\omega_{n}\tau-\eta_{n}}{2\pi}\right\Vert \right).
\end{align*}
This implies~(\ref{eq:sufficient-condition-u(t)}) and the assertion. 
\end{proof}
\begin{lem}
\label{lem:RN-bound}Let $N\geq1$ and $\gamma_{N}<T_{1}<\gamma_{N+1}$.
The function
\begin{equation}
R_{N}^{*}(y)=\sum_{n=N+1}^{\infty}\frac{\left|\rho_{0}\right|}{\left|\rho_{n}\right|}e^{-(\gamma_{n}-\gamma_{0})y},\qquad y>0.\label{eq:RN-star}
\end{equation}
satisfies
\[
R_{N}^{*}(y)<\frac{\left|\rho_{0}\right|e^{\gamma_{0}y}}{T_{1}e^{T_{1}y}}\left(\frac{1}{2\pi}\frac{\log\left(T_{1}\right)}{y}+4\log T_{1}+\frac{2}{T_{1}y}\right).
\]
If $T_{1}y\leq1$, then we also have
\[
R_{N}^{*}(y)<\frac{\left|\rho_{0}\right|e^{\gamma_{0}y}}{e^{T_{1}y}}\left(\frac{1}{4\pi}\left(\log^{2}y\right)+\frac{4\log T_{1}}{T_{1}}+\frac{2}{T_{1}}\right)+\frac{\left|\rho_{0}\right|e^{\gamma_{0}y-1}}{2\pi}\left|\log\left(y\right)\right|.
\]
\end{lem}
\begin{proof}
The series~(\ref{eq:RN-star}) is absolutely convergent when $y>0$.
It follows from~\cite[Lemma 1]{Lehman1966} that
\begin{align*}
R_{N}^{*}(y) & \leq\left|\rho_{0}\right|e^{\gamma_{0}y}\left(\frac{1}{2\pi}\int_{T_{1}}^{\infty}\frac{\log\left(\frac{t}{2\pi}\right)}{te^{ty}}dt+\frac{4\log T_{1}}{T_{1}e^{T_{1}y}}+2\int_{T_{1}}^{\infty}t^{-2}e^{-ty}dt\right)\\
 & =:\left|\rho_{0}\right|e^{\gamma_{0}y}\left(\frac{1}{2\pi}I_{1}+\frac{4\log T_{1}}{T_{1}e^{T_{1}y}}+2I_{2}\right)
\end{align*}
For every $T_{2}>T_{1}$ we have 
\begin{align*}
I_{1} & <e^{-T_{1}y}\int_{T_{1}}^{T_{2}}\frac{\log\left(t\right)}{t}dt+\frac{\log\left(T_{2}\right)}{T_{2}}\int_{T_{2}}^{\infty}e^{-ty}dt\\
 & =\frac{1}{2}e^{-T_{1}y}\left(\log^{2}T_{2}-\log^{2}T_{1}\right)+e^{-T_{2}y}\frac{\log\left(T_{2}\right)}{T_{2}y}.
\end{align*}
For $T_{2}=T_{1}$ we get
\[
I_{1}<e^{-T_{1}y}\frac{\log\left(T_{1}\right)}{T_{1}y}.
\]
If $T_{1}y\leq1$, we can choose $T_{2}=\frac{1}{y}$, so
\[
I_{1}<\frac{1}{2}e^{-T_{1}y}\log^{2}y+e^{-1}\left|\log\left(y\right)\right|.
\]
Moreover we have
\[
I_{2}<T_{1}^{-2}\int_{T_{1}}^{\infty}e^{-ty}dt=T_{1}^{-2}y^{-1}e^{-T_{1}y}
\]
and also
\[
I_{2}<e^{-T_{1}y}\int_{T_{1}}^{\infty}t^{-2}dt=T_{1}^{-1}e^{-T_{1}y}.
\]
\end{proof}
\begin{lem}
\label{lem:v-extremes}Let $0\leq\varphi'\leq\varphi\leq\varphi''\leq\pi$
and $\psi\in[0,\pi]$.
\begin{enumerate}
\item [$(i)$] If $\varphi''\leq\frac{\pi}{2}$, then
\[
v(\varphi,\psi)\leq\left(\cos\left(\frac{\varphi''-\varphi'}{2}\right)\right)^{-1}\max\left(v(\varphi',\psi),v(\varphi'',\psi)\right).
\]
\item [$(ii)$] If $\varphi'\geq\frac{\pi}{2}$, then
\[
v(\varphi,\psi)\leq v(\varphi',\psi).
\]
\item [$(iii)$] If $\varphi'\leq\frac{\pi}{2}\leq\varphi''$, then
\[
v(\varphi,\psi)\leq\left(\cos\left(\frac{\frac{\pi}{2}-\varphi'}{2}\right)\right)^{-1}\max\left(v(\varphi',\psi),v(\tfrac{\pi}{2},\psi)\right).
\]
\end{enumerate}
\end{lem}
\begin{proof}
$(ii)$ follows from
\[
\frac{\partial}{\partial\varphi}v(\varphi,\psi)=-\sin(\varphi)+\sin\left(\min(\varphi+\psi,\pi)\right)\leq0,\qquad\frac{\pi}{2}\leq\varphi\leq\pi,
\]
and $(iii)$ follows from $(i)$ and $(ii)$. To see $(i)$ note that
$v(\varphi,\psi)$ attains the value $\sup_{\varphi'\leq\varphi\leq\varphi''}v(\varphi,\psi)$
at the endpoints $\varphi=\varphi'$ or $\varphi=\varphi''$ or when
$\frac{\partial}{\partial\varphi}v(\varphi,\psi)=0$. In the first
two cases the assertion holds and the third one can only occur for
$\varphi=\frac{\pi-\psi}{2}$ when $\varphi'\leq\frac{\pi-\psi}{2}\leq\varphi''$.
In that case we have $\varphi+\psi\leq\pi$, so
\[
v(\varphi,\psi)=\cos(\varphi)-\cos(\varphi+\psi)=2\sin\left(\frac{\psi}{2}\right)\sin\left(\varphi+\frac{\psi}{2}\right)\leq2\sin\left(\frac{\psi}{2}\right).
\]
Similarly
\[
v(\varphi',\psi)=\cos(\varphi')-\cos(\varphi'+\psi)=2\sin\left(\frac{\psi}{2}\right)\sin\left(\varphi'+\frac{\psi}{2}\right)
\]
and
\[
v(\varphi'',\psi)\geq\cos(\varphi'')-\cos(\varphi''+\psi)=2\sin\left(\frac{\psi}{2}\right)\sin\left(\varphi''+\frac{\psi}{2}\right).
\]
We have
\begin{align*}
\max\left(\sin\left(\varphi'+\frac{\psi}{2}\right),\sin\left(\varphi''+\frac{\psi}{2}\right)\right) & =\max\left(\cos\left(\varphi'-\varphi\right),\cos\left(\varphi''-\varphi\right)\right)\\
 & =\cos\left(\min\left(\left|\varphi'-\varphi\right|,\left|\varphi''-\varphi\right|\right)\right)\\
 & \geq\cos\left(\frac{\varphi''-\varphi'}{2}\right).
\end{align*}
\end{proof}

\section{Diophantine approximation results\label{sec:Diophantine}}

\begin{algorithm}
\begin{pseudo}[label=\small\arabic*, line-height=1.2]*&

\textbf{$\!\!\!\!\!\!\!\!\!\!\!\!$Initial conditions:}

\\*&Input vectors $\boldsymbol{u}_{1},\dotsc,\boldsymbol{u}_{N}$
are linearly dependent over $\mathbb{R}$, but not over $\mathbb{Q}$

\\*&$c_{k,j}=1$ when $k\==j$, otherwise $c_{k,j}=0$ \ct{coefficients}

\\*&$\boldsymbol{v}_{k}=\sum_{j=1}^{N}c_{k,j}\boldsymbol{u}_{j}=\boldsymbol{u}_{k}$
for $k=1,\dotsc,N$

\\*&$\delta=\frac{1}{4}$

\\*&$M$ is a large integer \ct{coefficients bound}\vspace{0.6ex}

\\*&\textbf{$\!\!\!\!\!\!\!\!\!\!$Subroutine} \pr{Update}$(\ell,m)$:

\\ \kw{for} $a=1,\dotsc,\ell-1$

\\+		\kw{for} $b=\ell,\dotsc,m$

\\+			$\mu_{b,a}=\left(\boldsymbol{v}_{b}\cdot\boldsymbol{v}_{a}^{*}\right)/B_{a}$

\\-- \kw{for} $a=\ell,\dotsc,N$

\\+	$\boldsymbol{v}_{a}^{*}=\boldsymbol{v}_{a}-\sum_{b=1}^{a-1}\mu_{a,b}\boldsymbol{v}_{b}^{*}$

\\	$B_{a}=v_{a}^{*}\cdot v_{a}^{*}$

\\		\kw{for} $b=a+1,\dotsc,N$

\\+			$\mu_{b,a}=\left(\boldsymbol{v}_{b}\cdot\boldsymbol{v}_{a}^{*}\right)/B_{a}$\vspace{0.6ex}

\\--*&\textbf{$\!\!\!\!\!\!\!\!\!\!$Main Routine:}

\\ \pr{Update}$(1,N)$

\\$k=2$

\\ \kw{until} $k\==N+1$

\\+	\kw{for} $j=k-1,\dotsc,1$

\\+		$q=\lfloor\mu_{k,j}\rceil$

\\		\kw{if} $\max_{1\leq i\leq N}\left|c_{k,i}-qc_{j,i}\right|>M$

\\+			\kw{return} $\left(c_{k,j}\right)_{\substack{1\leq k\leq N\\
1\leq j\leq N
}
}$

\\-		\kw{for} $i=1,\dotsc,N$

\\+			$c_{k,i}=c_{k,i}-qc_{j,i}$

\\-		$\boldsymbol{v}_{k}=\sum_{i=1}^{N}c_{k,i}\boldsymbol{u}_{i}$

\\		\pr{Update}$(k,k)$

\\-	\kw{if} $B_{k}\geq\left(\delta-\mu_{k,k-1}^{2}\right)B_{k-1}$

\\+		$k=k+1$

\\-	\kw{else}

\\+		Swap $c_{k-1,i}$ and $c_{k,i}$ for $i=1,\dotsc,N$

\\		Swap $\boldsymbol{v}_{k-1}$ and $\boldsymbol{v}_{k}$

\\		\pr{Update}$(k-1,k)$

\\		$k=\max(2,k-1)$

\\-- \kw{return} $\left(c_{k,j}\right)_{\substack{1\leq k\leq N\\
1\leq j\leq N
}
}$

\end{pseudo}

\caption{\label{alg:LLL-bounded}Generating small integer linear combinations
of $u_{1},\dotsc,u_{N}$}

\end{algorithm}

In this section we prove Theorem~\ref{thm:Tiling-method}, but first
we explain how we construct the vectors $v_{k}$ needed for its application.
Let $\boldsymbol{u}_{1},\dotsc,\boldsymbol{u}_{N}$ denote the projections
of unit vectors $\boldsymbol{e}_{1},\dotsc,\boldsymbol{e}_{N}$, the
canonical basis of $\mathbb{R}^{N}$, onto the orthogonal space of
the vector $\left(\theta_{k}\right)_{k=1}^{N}$. Proposition~\ref{prop:Tiling}
below reduces the problem of constructing $\boldsymbol{v}_{k}$'s
to finding appropriate small linear combinations of $\boldsymbol{u}_{1},\dotsc,\boldsymbol{u}_{N}$,
with integer coefficients. We achieve the latter by applying Algorithm~\ref{alg:LLL-bounded}
to floating-point approximations of $\boldsymbol{u}_{1},\dotsc,\boldsymbol{u}_{N}$
(with 1000 decimal digits). As usual, $\lfloor x\rceil$ denotes the
integer closest to $x$, in particular $\left\Vert x\right\Vert =\left|x-\lfloor x\rceil\right|$.
Algorithm~\ref{alg:LLL-bounded} is a version of the LLL algorithm~\cite[p. 444]{HoffsteinPipherSilverman2014}
with the following differences:
\begin{itemize}
\item if the next reduction should result in coefficients greater than a
given limit $M$, the computation ends (Steps 14--15),
\item we keep track of the integer coefficients $c_{k,j}$; the reduced
vectors and the dependent data are re-computed from $c_{k,j}$ whenever
necessary, to prevent uncontrolled accumulation of round-off errors,
\item we use the value of $\delta=\frac{1}{4}$, which would not work well
with the standard LLL algorithm.
\end{itemize}
The procedure stops after reaching the coefficients bound or with
a division by zero attempt. In practice the latter does not occur.
We do not prove sufficient conditions for Algorithm~\ref{alg:LLL-bounded}
to succeed. Instead we use it as a heuristic tool to find the coefficients
$c_{k,j}$. It is in the computational part of the proof of Theorem~\ref{thm:Main-Theorem}
when we verify rigorously that our coefficients satisfy the requirements
of Proposition~\ref{prop:Tiling}. It is surprising that Algorithm~\ref{alg:LLL-bounded}
seems to work perfectly well with $\delta=\frac{1}{4}$, and in fact
the computation time is much shorter than for $\delta\geq\frac{3}{4}$,
required in the classical LLL algorithm. We did not study this phenomenon,
but we did make use of it. We proceed with the proofs of Proposition~\ref{prop:Tiling}
and Theorem~\ref{thm:Tiling-method}.
\begin{prop}
\label{prop:Tiling}Let $\boldsymbol{\eta},\boldsymbol{\theta}\in\mathbb{R}^{N}$,
$\boldsymbol{\theta}=(\theta_{1},\dotsc,\theta_{N})\neq(0,\dotsc,0)$,
and let $A\subseteq\left[-\tfrac{1}{2},\tfrac{1}{2}\right]^{N}$ be
a compact set. Put
\[
\boldsymbol{u}_{j}=\boldsymbol{e}_{j}-\mu_{j}\boldsymbol{\theta},\qquad\mu_{j}=\frac{\theta_{j}}{\boldsymbol{\theta}\cdot\boldsymbol{\theta}},\qquad j=1,\dotsc,N.
\]
Let $C=\left(c_{k,j}\right)_{\substack{1\leq k\leq N\\
1\leq j\leq N
}
}$ be a non-singular $N\times N$ matrix with integer coefficients.
Let $d_{1},\dotsc,d_{N}$ be positive numbers such that for every
$(\epsilon_{k})_{k=1}^{N}\in\left\{ -1,1\right\} ^{N}$ we have
\begin{equation}
\sum_{k=1}^{N}\epsilon_{k}\sum_{j=1}^{N}c_{k,j}\boldsymbol{u}_{j}\in D:=\prod_{n=1}^{N}(-d_{n},d_{n}).\label{eq:vertices-inside-open-D}
\end{equation}
Let
\[
A'=\left\{ \boldsymbol{v}\in A:\boldsymbol{v}+\frac{3}{2}D\subseteq A\right\} .
\]
Then for every $\delta>0$ there exists an increasing sequence of
non-negative numbers $(\tau_{\ell})_{\ell=1}^{\infty}$ such that
for all $\ell$ we have $\left(\theta_{n}\tau_{\ell}\right)_{n=1}^{N}\in\boldsymbol{\eta}+A+\mathbb{Z}^{N}$
and $\tau_{\ell+1}\geq\tau_{\ell}+\delta$, and 
\[
\liminf_{\ell\to\infty}\frac{\ell}{\tau_{\ell}}\geq\frac{1}{\delta}\operatorname{vol}(A').
\]
\end{prop}
\begin{proof}
Let $\boldsymbol{v}_{k}=\sum_{j=1}^{N}c_{k,j}\boldsymbol{u}_{j}$
for $k=1,\dotsc,N$. Since $C$ is non-singular, the vectors $\boldsymbol{v}_{k}$,
$k=1,\dotsc,N$, span the same $(N-1)$-dimensional subspace as the
vectors $\boldsymbol{u}_{j}$, $j=1,\dotsc,N$, i.e. the orthogonal
space of $\boldsymbol{\varTheta}$. Hence we can choose $N-1$ linearly
independent vectors from the $\boldsymbol{v}_{k}$'s, say $\boldsymbol{v}_{1},\dotsc,\boldsymbol{v}_{N-1}$,
possibly after renumbering. We replace $\boldsymbol{v}_{N}$ with
$\varepsilon\boldsymbol{\varTheta}$, for any $\varepsilon>0$, and
we obtain a basis $\boldsymbol{v}_{1},\dotsc,\boldsymbol{v}_{N}$
of $\mathbb{R}^{N}$. Let $\varepsilon$ be sufficiently small, so
that $\varepsilon\max_{n}\left|\theta_{n}\right|$ is less than the
smallest distance of $\sum_{k=1}^{N}\epsilon_{k}\sum_{j=1}^{N}c_{k,j}\boldsymbol{u}_{j}$
to the boundary of $D$ in~(\ref{eq:vertices-inside-open-D}) (here
$D$ is open). Then we have
\[
\pm\boldsymbol{v}_{1}\pm\dotsc\pm\boldsymbol{v}_{N}\in D.
\]
Since
\[
\boldsymbol{v}_{k}=\sum_{j=1}^{N}c_{k,j}\left(\boldsymbol{e}_{j}-\mu_{j}\boldsymbol{\theta}\right)\in-\left(\sum_{j=1}^{N}c_{k,j}\mu_{j}\right)\boldsymbol{\theta}+\mathbb{Z}^{N},\qquad k=1,\dotsc,N-1,
\]
and $\boldsymbol{v}_{N}=\varepsilon\boldsymbol{\theta}$, the assertion
follows by Theorem~\ref{thm:Tiling-method}.
\end{proof}

\subsection*{Proof of Theorem~\ref{thm:Tiling-method}}

Let $\boldsymbol{\theta}=(\theta_{n})_{n=1}^{N}$ again, moreover
we take $T>0$ large and $\varepsilon>0$ small. Let $P$ be as in(\ref{eq:primitive-cell})
and $P'=(1-\varepsilon)P$. We can assume $d_{n}<\frac{1}{3}$ for
$n=1,\dotsc,N$, as otherwise $\operatorname{vol}(A')=0$. Hence $P'\subseteq\left(-\tfrac{1}{6},\tfrac{1}{6}\right)^{N}$.
We have
\begin{multline*}
\int_{\boldsymbol{v}\in(\mathbb{R}/\mathbb{Z})^{N}}\left|\left\{ t\in[0,T]:t\boldsymbol{\theta}\in\boldsymbol{v}+P'+\mathbb{Z}^{N}\right\} \right|\,d\boldsymbol{v}=\\
=\int_{0}^{T}\operatorname{vol}\left\{ \boldsymbol{v}\in(\mathbb{R}/\mathbb{Z})^{N}:t\boldsymbol{\theta}\in\boldsymbol{v}+P'+\mathbb{Z}^{N}\right\} \,dt\\
=T\operatorname{vol}\left(P'\right),
\end{multline*}
where $\left|\cdot\right|$ denotes the 1-dimensional Lebesgue measure.
Therefore there exists a $\tilde{\boldsymbol{v}}\in\mathbb{R}^{N}$
such that the set
\[
\mathcal{T}=\left\{ t\in[0,T]:t\boldsymbol{\theta}\in\tilde{\boldsymbol{v}}+P'+\mathbb{Z}^{N}\right\} 
\]
satisfies $\left|\mathcal{T}\right|\geq T\operatorname{vol}\left(P'\right)$.
The set $\mathcal{T}$ is closed, as a preimage of a closed set under
a continuous mapping, so let $t_{0}=\min\mathcal{T}$ and $\boldsymbol{v}_{0}=\tilde{\boldsymbol{v}}-t_{0}\boldsymbol{\theta}$.
We have 
\[
\boldsymbol{v}_{0}\in\tilde{\boldsymbol{v}}-(\tilde{\boldsymbol{v}}+P'+\mathbb{Z}^{N})=P'+\mathbb{Z}^{N}.
\]
Consider the set of lattice points deep inside the set $\boldsymbol{\eta}+A$,
and the corresponding shifts of $\mathcal{T}$:
\[
L=\left\{ \boldsymbol{v}=\sum_{n=1}^{N}\lambda_{n}\boldsymbol{v}_{n}:\boldsymbol{v}+D\subseteq\boldsymbol{\eta}+A\text{ and }\lambda_{n}\in\mathbb{Z}\text{ for }n=1,\dotsc,N\right\} ,
\]
\[
\mathcal{T}_{\boldsymbol{v}}=-t_{0}+\sum_{n=1}^{N}\lambda_{n}t_{n}+\mathcal{T},\qquad\boldsymbol{v}=\sum_{n=1}^{N}\lambda_{n}\boldsymbol{v}_{n}\in L.
\]
For every $t'\in\mathcal{T}_{\boldsymbol{v}}$, $\boldsymbol{v}=\sum_{n=1}^{N}\lambda_{n}\boldsymbol{v}_{n}$,
we have $t'=t-t_{0}+\sum_{n=1}^{N}\lambda_{n}t_{n}$ for some $t\in\mathcal{T}$,
so 
\[
t'\boldsymbol{\theta}=\boldsymbol{v}+t\boldsymbol{\theta}-t_{0}\boldsymbol{\theta}\in\boldsymbol{v}+\tilde{\boldsymbol{v}}-t_{0}\boldsymbol{\theta}+P'+\mathbb{Z}^{N}=\boldsymbol{v}+\boldsymbol{v}_{0}+P'+\mathbb{Z}^{N}.
\]
Since the sets $\boldsymbol{v}+\boldsymbol{v}_{0}+P'+\mathbb{Z}^{N}$,
$\boldsymbol{v}\in L$, are pairwise disjoint, so are the sets $\mathcal{T}_{\boldsymbol{v}}$,
so their union satisfies
\[
\left|\bigcup_{\boldsymbol{v}\in L}\mathcal{T}_{\boldsymbol{v}}\right|\geq T\operatorname{vol}\left(P'\right)\cdot\#L.
\]
We need to estimate $\#L$ from below. For $\boldsymbol{x}\in\boldsymbol{\eta}+A'$
let $\boldsymbol{v}=\sum_{n=1}^{N}\lambda_{n}\boldsymbol{v}_{n}$
be such that $\lambda_{n}\in\mathbb{Z}$ for $n=1,\dotsc,N$, and
$\boldsymbol{x}\in\boldsymbol{v}+P$. Then $\boldsymbol{v}\in\boldsymbol{x}+P\subseteq\boldsymbol{x}+\tfrac{1}{2}D$,
so $\boldsymbol{v}+D\subseteq\boldsymbol{x}+\tfrac{3}{2}D\subseteq\boldsymbol{\eta}+A$.
Hence $\boldsymbol{v}\in L$. Since we can choose $\boldsymbol{x}\in\boldsymbol{\eta}+A'$
arbitrarily, we obtain
\[
\boldsymbol{\eta}+A'\subseteq\bigcup_{\boldsymbol{v}\in L}\left(\boldsymbol{v}+P\right),
\]
hence $\#L\geq\operatorname{vol}\left(A'\right)/\operatorname{vol}\left(P\right)$.
Let $M$ denote the maximum of $\left|\sum_{n=1}^{N}\lambda_{n}t_{n}\right|$
over $\left(\lambda_{n}\right)\in\mathbb{Z}^{N}$ such that $\sum_{n=1}^{N}\lambda_{n}\boldsymbol{v}_{n}\in L$,
and let
\[
\mathcal{T}'=\left\{ t\in\mathbb{R}:t\boldsymbol{\theta}\in\boldsymbol{\eta}+A+\mathbb{Z}^{N}\right\} .
\]
Note that $\mathcal{T}'$ is also closed and it does not depend on
the choice of $T$ and $\varepsilon$. We have $\bigcup_{\boldsymbol{v}\in L}\mathcal{T}_{\boldsymbol{v}}\subseteq\mathcal{T}'\cap[-M,T+M]$,
so
\begin{align*}
\left|\mathcal{T}'\cap[-M,T+M]\right| & \geq\frac{\operatorname{vol}\left(P'\right)}{\operatorname{vol}\left(P\right)}T\operatorname{vol}\left(A'\right)\\
 & =\left(1-\varepsilon\right)^{N}T\operatorname{vol}\left(A'\right)
\end{align*}
for every $\varepsilon>0$. Hence $\left|\mathcal{T}'\cap[-M,T+M]\right|\geq T\operatorname{vol}\left(A'\right)$.
Define $\tau_{1}=\min\left(\mathcal{T}'\cap[0,+\infty)\right)$ and
$\tau_{\ell+1}=\min\left(\mathcal{T}'\cap[\tau_{\ell}+\delta,+\infty)\right)$,
$\ell\geq1$. Then
\[
\mathcal{T}'\cap[0,T]\subseteq\bigcup_{\tau_{\ell}\leq T}[\tau_{\ell},\tau_{\ell}+\delta],
\]
so
\begin{align*}
\#\left\{ \ell:\tau_{\ell}\leq T\right\}  & \geq\frac{1}{\delta}\left|\mathcal{T}'\cap[0,T]\right|\\
 & \geq\frac{T}{\delta}\left(\operatorname{vol}\left(A'\right)-\frac{2M}{T}\right).
\end{align*}
This implies the assertion. \qed 

\section{The measure of the solution set}

Next we describe the procedure of estimating the measure of the set
$A'$ defined in Theorem~\ref{thm:Tiling-method} and Proposition~\ref{prop:Tiling}
when $A$ is defined by~(\ref{eq:solution-set}) and $w_{n}:\left[0,\frac{1}{2}\right]\to[0,1]$,
$n=1,\dotsc,N$, are arbitrary non-decreasing functions.
\begin{lem}
\label{lem:volume}Let $N$ and $\ell$ be positive integers. For
$n=1,\dotsc,N$ let $\varepsilon_{n}>0$ and let $w_{n}:\left[0,\tfrac{1}{2}\right]\to[0,1]$
be non-decreasing functions. Let $A$ be as in~(\ref{eq:solution-set})
and let 
\[
A'=\left\{ \boldsymbol{v}\in A:\boldsymbol{v}+\prod_{n=1}^{N}\left[-\varepsilon_{n},\varepsilon_{n}\right]\subseteq A\right\} .
\]
For each $n=1,\dotsc,N$ let $0=x_{n,0}\leq x_{n,1}\leq\dotsc\leq x_{n,\ell}\leq\tfrac{1}{2}-\varepsilon_{n}$
be such that
\[
w_{n}\left(\varepsilon_{n}+x_{n,i}\right)\leq\frac{i}{\ell},\qquad i=1,\dotsc,\ell.
\]
Let
\[
r_{i}=\underset{k_{1}+\dotsc+k_{N}=i}{\sum_{k_{1}=1}^{\ell}\dotsc\sum_{k_{N}=1}^{\ell}}\prod_{n=1}^{N}\left(x_{n,k_{n}}-x_{n,k_{n}-1}\right),\qquad i=1,\dotsc,\ell.
\]
Then
\begin{equation}
\operatorname{vol}(A')\geq2^{N}\sum_{i=1}^{\ell}r_{i}.\label{eq:volume}
\end{equation}
\end{lem}
\begin{proof}
Let $\boldsymbol{k}=(k_{n})_{n=1}^{N}$ be such that $1\leq k_{n}\leq\ell$
for each $n$ and $\sum_{n=1}^{N}k_{n}=i\leq\ell$. The hyperrectangle
$P_{\boldsymbol{k}}=\prod_{n=1}^{N}\left[x_{n,k_{n}-1},x_{n,k_{n}}\right]$
is contained in $A'$, because for every $\boldsymbol{v}=\left(v_{n}\right)_{n=1}^{N}\in P_{\boldsymbol{k}}$
and $\boldsymbol{\delta}=\left(\delta_{n}\right)_{n=1}^{N}\in\prod_{n=1}^{N}\left[-\varepsilon_{n},\varepsilon_{n}\right]$
we have $\boldsymbol{v}+\boldsymbol{\delta}\in\left[-\tfrac{1}{2},\tfrac{1}{2}\right]^{N}$
and
\[
\sum_{n=1}^{N}w_{n}\left(\left|v_{n}+\delta_{n}\right|\right)\leq\sum_{n=1}^{N}w_{n}\left(v_{n}+\varepsilon_{n}\right)\leq\sum_{n=1}^{N}\frac{k_{n}}{\ell}=\frac{i}{\ell}\leq1.
\]
By the symmetry of $A'$ all the reflections of $P_{\boldsymbol{k}}$
relative to hyperplanes spanned by axes are also contained in $A'$.
This implies the assertion.
\end{proof}

\section{Proof of Theorem~\ref{thm:Main-Theorem}\label{sec:Proof}}

The proof contains several numbered claims whose validity we verify
by rigorous machine computation in Mathematica using CenteredInterval
objects that bundle the ball arithmetic from the ARB library~\cite{ARB}.
The source code for this verification, as well as a demonstration
of how we selected the final parameter values used in this proof,
are available for inspection in the form of Mathematica notebooks~\cite{GKPR}.
We quote here all the data chosen arbitrarily, such as the definition
of the contour $z_{1}(t)$, so that the reader can, in principle,
recover our computations from the paper text alone. We have used the
high-precision estimates of the zeros of the Riemann Zeta function
from the Flint library~\cite{FLINT}, which now incorporates ARB.
We have checked that they agree with those returned by Mathematica
and, in case of the initial zeros, with those published by Odlyzko~\cite{Odlyzko}.

It was proved in~\cite[Theorem 2]{Kaczorowski1985} that $\kappa_{0}\geq\gamma^{*}/\pi$,
where $\gamma^{*}>0$ denotes the ordinate of the lowest-lying non-trivial
zero of the Riemann zeta-function on the line $\sigma=\sup_{\zeta(\rho)=0}\Re(\rho)$
if there are any such zeros. Otherwise $\gamma^{*}=\infty$. Since
all the non-trivial zeros with $|\gamma|\leq3\cdot10^{12}$ lie on
the critical line (cf.~\cite{PlattTrudgian2021}), we deduce that
\[
\kappa_{0}\geq\frac{3}{\pi}\cdot10^{12}>\frac{\gamma_{0}}{\pi}+9.54\cdot10^{11}
\]
if the Riemann Hypothesis is not true. This is much stronger than
the assertion of Theorem~\ref{thm:Main-Theorem}. Hence, in the rest
of the proof, we can assume that the Riemann Hypothesis is true. 

Let $N=21$, $\boldsymbol{\eta}=(\eta_{1},\dotsc,\eta_{N})=\left(\pi-\arg\left(\rho_{0}/\rho_{n}\right)\right)_{n=1}^{N}$
and fix $F=F^{*}$, in particular $a(n)=\frac{\rho_{0}}{\rho_{n}}$,
$\omega_{n}=\gamma_{n}-\gamma_{0}$ (for the initial zeros, that are
confirmed to be simple) and
\[
F_{N}^{\boldsymbol{\eta}}(z)=1-\sum_{n=1}^{N}\left|\frac{\rho_{0}}{\rho_{n}}\right|e^{i\omega_{n}z}.
\]
Let $f(z)$ denote the derivative of $F_{N}^{\boldsymbol{\eta}}(z)$,
i.e.
\[
f(z)=-\sum_{n=1}^{N}i\omega_{n}\left|\frac{\rho_{0}}{\rho_{n}}\right|e^{i\omega_{n}z}.
\]
Define the contour $z_{1}(t)$ as
\[
z_{1}(t)=\begin{cases}
4t\delta+Y_{0}i, & 0\leq t\leq\tfrac{1}{8},\\
\tfrac{1}{2}\delta+\left(Y_{0}+4(t-\tfrac{1}{8})(Y_{1}-Y_{0})\right)i, & \tfrac{1}{8}\leq t\leq\tfrac{3}{8},\\
-4(t-\tfrac{1}{2})\delta+Y_{1}i, & \tfrac{3}{8}\leq t\leq\tfrac{5}{8},\\
-\tfrac{1}{2}\delta+\left(Y_{1}+4(t-\tfrac{5}{8})(Y_{0}-Y_{1})\right)i & \tfrac{5}{8}\leq t\leq\tfrac{7}{8},\\
4(t-1)\delta+Y_{0}i, & \tfrac{7}{8}\leq t\leq1,
\end{cases}
\]
where $\delta=\mathtt{0.132737}$, $Y_{0}=\mathtt{0.0468918}$, $Y_{1}=\mathtt{0.14}$.
Let $m=\mathtt{12000}$ and $t_{j}=\frac{j}{m}$, $j=0,\dotsc,m$.
Note that $z_{1}([t_{j-1},t_{j}])$ is a straight line segment and
the derivative of $z_{1}(t)$ is constant inside this interval for
$j=1,\dotsc,m$. Let $\alpha(t)$ be a continous piecewise-linear
function defined on $[0,\frac{1}{2}]$ by $\alpha(0)=0$, $\alpha(\frac{1}{32})=\mathtt{0.489819}$,
$\alpha(\frac{1}{8})={\tt 1.85802}$, $\alpha(\frac{3}{16})={\tt 2.04829}$,
$\alpha(\frac{3}{8})={\tt 2.90189}$, $\alpha(\frac{1}{2})=\pi$ and
changing linearly in-between those points. We define $\left(\alpha_{j}\right)_{j=0}^{m+1}$
by
\[
\alpha_{j}=\begin{cases}
-\pi+\alpha\left(\frac{j-1/2}{m}\right) & 1\leq j\leq\frac{m}{2},\\
\pi-\alpha\left(\frac{m+1/2-j}{m}\right), & \frac{m}{2}+1\leq j\leq m,
\end{cases}
\]
moreover $\alpha_{m+1}=2\pi+\alpha_{1}$ and $\alpha_{0}=\alpha_{m}-2\pi$.

\begin{claim}
\label{claim:arguments}We have $\left|\alpha_{j}-\alpha_{j-1}\right|<\pi$
and $\left|z_{1}(t_{j})-z_{1}(t_{j-1})\right|<\frac{\pi}{\omega_{N}}$
for $j=1,\dotsc,m$.
\end{claim}
Define 
\[
z_{2}(t)=\exp\left(i\left(\frac{\alpha_{j-1}+\alpha_{j}}{2}+\frac{t-t_{j-1}}{t_{j}-t_{j-1}}\cdot\frac{\alpha_{j+1}-\alpha_{j-1}}{2}\right)\right),\qquad t_{j-1}\leq t\leq t_{j},j=1,\dotsc,m,
\]
and
\[
\eta(t)=\alpha_{j}-\left(\frac{\alpha_{j-1}+\alpha_{j}}{2}+\frac{t-t_{j-1}}{t_{j}-t_{j-1}}\cdot\frac{\alpha_{j+1}-\alpha_{j-1}}{2}\right),\qquad t_{j-1}\leq t\leq t_{j},j=1,\dotsc,m,
\]
so that
\[
F_{N}^{\boldsymbol{\eta}}(z_{1}(t))\overline{z_{2}(t)}e^{-i\eta(t)}=F_{N}^{\boldsymbol{\eta}}(z_{1}(t))e^{-i\alpha_{j}}
\]
and
\[
\varphi_{n}(t)=\eta_{n}+\arg a(n)+\omega_{n}\Re z_{1}(t)-\arg z_{2}(t)-\eta(t)=\pi+\omega_{n}\Re z_{1}(t)-\alpha_{j}
\]
for $t_{j-1}\leq t\leq t_{j}$, $j=1,\dotsc,m$ and $n=1,\dotsc,N$.
It follows from Claim~\ref{claim:arguments} that $z_{1}(t)$ and
$z_{2}(t)$ are closed contours and that $\left|\eta(t)\right|\leq\frac{\pi}{2}$
for $0\leq t\leq1$. For $j=1,\dotsc,m$ let 
\[
b_{j}=\tfrac{1}{2}\left|z_{1}(t_{j-1})-z_{1}(t_{j})\right|,
\]
\[
x_{j}'=\min(\Re z_{1}(t_{j-1}),\Re z_{1}(t_{j})),
\]
\[
x_{j}''=\max(\Re z_{1}(t_{j-1}),\Re z_{1}(t_{j})),
\]
\[
y_{j}=\min(\Im z_{1}(t_{j-1}),\Im z_{1}(t_{j})),
\]
\[
M_{j}=\sum_{n=1}^{N}\left|\frac{\rho_{0}}{\rho_{n}}\right|\omega_{n}^{2}e^{-\omega_{n}y_{j}}.
\]
We have $\left|f'(z)\right|\leq M_{j}$ for $t_{j-1}\leq t\leq t_{j}$,
$j=1,\dotsc,m$. Let $N'=9998$ , $T_{1}=\frac{\gamma_{N'}+\gamma_{N'+1}}{2}$
and define
\[
\widetilde{R}_{N}(y)=\sum_{n=N+1}^{N'}\left|\frac{\rho_{0}}{\rho_{n}}\right|e^{-\omega_{n}y}+\frac{\left|\rho_{0}\right|e^{\gamma_{0}y}}{T_{1}e^{T_{1}y}}\left(\frac{1}{2\pi}\frac{\log\left(T_{1}\right)}{y}+4\log T_{1}+\frac{2}{T_{1}y}\right),\qquad y>0.
\]
It follows from Lemma~\ref{lem:RN-bound} that $\widetilde{R}_{N}(y)>R_{N}^{*}(y)$
for $y>0$. Moreover the function $R_{N}(y)$ is monotone decreasing
in $y$. For $t_{j-1}\leq t\leq\frac{t_{j-1}+t_{j}}{2}$ we have
\[
\Re\left(F_{N}^{\boldsymbol{\eta}}(z_{1}(t))e^{-i\alpha_{j}}\right)\geq\Re\left(F_{N}^{\boldsymbol{\eta}}(z_{1}(t_{j-1}))e^{-i\alpha_{j}}\right)+\min\left(0,\Re\left(b_{j}f(z_{1}(t_{j-1}))e^{-i\alpha_{j}}\right)\right)-\frac{b_{j}^{2}M_{j}}{2}.
\]
Likewise
\[
\Re\left(F_{N}^{\boldsymbol{\eta}}(z_{1}(t))e^{-i\alpha_{j}}\right)\geq\Re\left(F_{N}^{\boldsymbol{\eta}}(z_{1}(t_{j}))e^{-i\alpha_{j}}\right)+\min\left(0,-\Re\left(b_{j}f(z_{1}(t_{j}))e^{-i\alpha_{j}}\right)\right)\Bigr)-\frac{b_{j}^{2}M_{j}}{2}
\]
for $\frac{t_{j-1}+t_{j}}{2}\leq t\leq t_{j}$, $j=1,\dotsc,m$. Therefore
the numbers
\begin{align*}
u_{j}= & \min\Bigl(\Re\left(F_{N}^{\boldsymbol{\eta}}(z_{1}(t_{j-1}))e^{-i\alpha_{j}}\right)+\min\left(0,\Re\left(b_{j}f(z_{1}(t_{j-1}))e^{-i\alpha_{j}}\right)\right),\\
 & \Re\left(F_{N}^{\boldsymbol{\eta}}(z_{1}(t_{j}))e^{-i\alpha_{j}}\right)+\min\left(0,-\Re\left(b_{j}f(z_{1}(t_{j}))e^{-i\alpha_{j}}\right)\right)\Bigr)\\
 & -\frac{b_{j}^{2}M_{j}}{2}-\widetilde{R}_{N}(y_{j}),\qquad j=1,\dotsc,m,
\end{align*}
satisfy
\[
u_{j}<\Re\left(F_{N}^{\boldsymbol{\eta}}(z_{1}(t))e^{-i\alpha_{j}}\right)-R_{N}(\Im z_{1}(t)),\qquad t_{j-1}\leq t\leq t_{j},
\]
for all $j$.

\begin{claim}
We have $u_{j}>0$ for $j=1,\dotsc,m$.
\end{claim}
For $n=1,\dotsc,N$ and $j=1,\dotsc,m$ let
\[
\beta'_{n,j}=\pi+\omega_{n}x_{j}'-\alpha_{j},
\]
\[
\beta''_{n,j}=\pi+\omega_{n}x_{j}''-\alpha_{j},
\]
moreover, if
\begin{equation}
\beta'_{n,j}\leq2\ell\pi\leq\beta''_{n,j}\hspace{1em}\text{for some }\ell\in\mathbb{Z},\label{eq:beta-condition-2pi}
\end{equation}
then
\[
\varphi'_{n,j}=0,\qquad\varphi''_{n,j}=2\pi\max\left(\left\Vert \frac{\beta'_{n,j}}{2\pi}\right\Vert ,\left\Vert \frac{\beta''_{n,j}}{2\pi}\right\Vert \right);
\]
if
\begin{equation}
\beta'_{n,j}\leq(2\ell+1)\pi\leq\beta''_{n,j}\hspace{1em}\text{for some }\ell\in\mathbb{Z},\label{eq:beta-condition-pi}
\end{equation}
then
\[
\varphi'_{n,j}=2\pi\min\left(\left\Vert \frac{\beta'_{n,j}}{2\pi}\right\Vert ,\left\Vert \frac{\beta''_{n,j}}{2\pi}\right\Vert \right),\qquad\varphi''_{n,j}=\pi;
\]
otherwise, if~(\ref{eq:beta-condition-2pi}) and~(\ref{eq:beta-condition-pi})
are false, then
\[
\varphi'_{n,j}=2\pi\min\left(\left\Vert \frac{\beta'_{n,j}}{2\pi}\right\Vert ,\left\Vert \frac{\beta''_{n,j}}{2\pi}\right\Vert \right),\qquad\varphi''_{n,j}=2\pi\max\left(\left\Vert \frac{\beta'_{n,j}}{2\pi}\right\Vert ,\left\Vert \frac{\beta''_{n,j}}{2\pi}\right\Vert \right).
\]
Note that conditions (\ref{eq:beta-condition-2pi}) and~(\ref{eq:beta-condition-pi})
are mutually exclusive by the following claim.

\begin{claim}
We have $\beta''_{n,j}-\beta'_{n,j}<\pi$ for $n=1,\dotsc,N$ and
$j=1,\dotsc,m$.
\end{claim}
We have
\[
\varphi'_{n,j}\leq2\pi\left\Vert \frac{\varphi_{n}(t)}{2\pi}\right\Vert \leq\varphi''_{n,j},\qquad t_{j-1}\leq t\leq t_{j},\;j=1,\dotsc,m,\;n=1,\dotsc,N.
\]
It follows from Lemma~\ref{lem:v-extremes} that
\[
\sup_{0\leq t\leq1}\frac{\left|a(n)\right|}{u_{j}e^{\omega_{n}\Im z_{1}(t)}}v\left(2\pi\left\Vert \frac{\varphi_{n}(t)}{2\pi}\right\Vert ,2\pi x\right)\leq\max_{j=0,\dotsc,m}c_{n,j}\max\left(v\left(\varphi'_{n,j},2\pi x\right),v\left(\varphi''_{n,j},2\pi x\right)\right),
\]
where
\[
c_{n,j}=\begin{cases}
\left|a(n)\right|u_{j}^{-1}e^{-\omega_{n}y_{j}}\left(\cos\left(\frac{\varphi''_{n,j}-\varphi'_{n,j}}{2}\right)\right)^{-1}, & \varphi''_{n,j}\leq\frac{\pi}{2},\\
\left|a(n)\right|u_{j}^{-1}e^{-\omega_{n}y_{j}} & \varphi'_{n,j}\geq\frac{\pi}{2},\\
\left|a(n)\right|u_{j}^{-1}e^{-\omega_{n}y_{j}}\left(\cos\left(\frac{\frac{\pi}{2}-\varphi'_{n,j}}{2}\right)\right)^{-1}, & \varphi'_{n,j}<\frac{\pi}{2}<\varphi''_{n,j},
\end{cases}
\]
for $j=1,\dotsc,m$, moreover $\varphi'_{n,0}=\varphi''_{n,0}=\frac{\pi}{2}$
and
\[
c_{n,0}=\max_{\substack{\varphi'_{n,j}\leq\frac{\pi}{2}\\
\varphi''_{n,j}\geq\frac{\pi}{2}
}
}\left|a(n)\right|u_{j}^{-1}e^{-\omega_{n}y_{j}}\left(\cos\left(\frac{\frac{\pi}{2}-\varphi'_{n,j}}{2}\right)\right)^{-1}
\]
if the last $\max$ is over a non-empty set, otherwise $c_{n,0}=0$.
We define
\[
w_{n}(x)=\max_{j=0,\dotsc,m}c_{n,j}\max\left(v\left(\varphi'_{n,j},2\pi x\right),v\left(\varphi''_{n,j},2\pi x\right)\right),\qquad0\leq x\leq\frac{1}{2},n=1,\dotsc,N.
\]
Let
\[
A=\left\{ (v_{1},\dotsc,v_{N}):\sum_{n=1}^{N}w_{n}\left(\left|v_{n}\right|\right)\leq1\right\} .
\]
For every $\tau\in\mathbb{R}$ such that 
\[
\left(\frac{\omega_{n}\tau}{2\pi}\right)_{n=1}^{N}\in\frac{\boldsymbol{\eta}}{2\pi}+A+\mathbb{Z}^{N}
\]
we have
\[
\sum_{n=1}^{N}w_{n}\left(\left\Vert \frac{\omega_{n}\tau-\eta_{n}}{2\pi}\right\Vert \right)\leq1,
\]
so it follows from Proposition~\ref{prop:sufficient-condition} that
there is a zero of $F(z)$ inside the contour $z_{1}(t)+\tau$. Let
\[
\boldsymbol{\theta}=(\theta_{n})_{n=1}^{N}=\left(\frac{\omega_{n}}{2\pi}\right)_{n=1}^{N}
\]
and
\[
\boldsymbol{u}_{j}=\left(u_{j,n}\right)_{n=1}^{N}=\boldsymbol{e}_{j}-\mu_{j}\boldsymbol{\theta},\qquad\mu_{j}=\frac{\theta_{j}}{\boldsymbol{\theta}\cdot\boldsymbol{\theta}},\qquad j=1,\dotsc,N.
\]
We apply Algorithm~\ref{alg:LLL-bounded} to $\left(\boldsymbol{u}_{j}\right)_{j=1}^{N}$
with coefficients bound $10^{300}$ and obtain the following.

\begin{claim}
\label{claim:Algorithm1}There exists a matrix $C=\left(c_{k,j}\right)_{\substack{1\leq k\leq N\\
1\leq j\leq N
}
}$ with integer coefficients such that
\begin{equation}
\sum_{k=1}^{N}\left|\sum_{j=1}^{N}c_{k,j}u_{j,n}\right|<1.66\cdot10^{-13},\qquad n=1,\dotsc,N.\label{eq:dn-ineq}
\end{equation}
\end{claim}
It follows from Proposition~\ref{prop:Tiling} that there exists
an increasing sequence of non-negative numbers $(\tau_{i})_{i=1}^{\infty}$
such that for all $i$ we have $\left(\theta_{n}\tau_{i}\right)_{n=1}^{N}\in\boldsymbol{\eta}+A+\mathbb{Z}^{N}$
and $\tau_{i+1}\geq\tau_{i}+\delta$, and 
\[
\liminf_{i\to\infty}\frac{\ell}{\tau_{i}}\geq\frac{1}{\delta}\operatorname{vol}(A'),
\]
where $A'$ is as defined in Proposition~\ref{prop:Tiling}. Here
$\delta$ is, as above, the width of the contour $z_{1}(t)$. We obtain
a sequence of zeros of $F(z)$ inside the shifted contours $z_{1}(t)+\tau_{\ell}$.
By $\tau_{\ell+1}\geq\tau_{\ell}+\delta$ these are distinct zeros.
Hence $\varkappa\geq\frac{1}{\delta}\operatorname{vol}(A')$. Now
we need to show a lower bound for $\operatorname{vol}(A')$. We take
$\varepsilon=2\cdot1.66\cdot10^{-13}$, i.e. twice the right-hand
side of~(\ref{eq:dn-ineq}), choose $\ell=\mathtt{16000}$, and find
\begin{equation}
0\leq x_{n,1}\leq\dotsc\leq x_{n,\ell}\leq\tfrac{1}{2}-\varepsilon\label{eq:x-inequalities}
\end{equation}
satisfying 
\begin{equation}
w_{n}\left(\varepsilon+x_{n,k}\right)\leq\frac{k}{\ell},\qquad k=1,\dotsc,\ell,n=1,\dotsc,N.\label{eq:wn-inequalities}
\end{equation}
Then we use~(\ref{eq:volume}) to estimate the volume.

\begin{claim}
\label{claim:measure}There exist $x_{n,k}$ satisfying~(\ref{eq:x-inequalities}),~(\ref{eq:wn-inequalities})
and
\[
\frac{2}{\delta}\cdot2^{N}\underset{k_{1}+\dotsc+k_{N}\leq\ell}{\sum_{k_{1}=1}^{\ell}\dotsc\sum_{k_{N}=1}^{\ell}}\prod_{n=1}^{N}\left(x_{n,k_{n}}-x_{n,k_{n}-1}\right)\geq\frac{1}{60},
\]
where $x_{n,0}=0$ for $n=1,\dotsc,N$.
\end{claim}
By Claim~\ref{claim:measure} and Lemma~\ref{lem:volume} we have
$2\varkappa\geq\frac{2}{\delta}\operatorname{vol}(A')>\frac{1}{60}$.\qed

\bigskip
\bigskip
\noindent
Maciej Grześkowiak, Faculty of Mathematics and Computer Science, Adam Mickiewicz University, Pozna\'n, 61-614 Pozna\'n, Poland. e-mail: \url{maciejg@amu.edu.pl}

\medskip
\noindent
Jerzy Kaczorowski, Faculty of Mathematics and Computer Science, Adam Mickiewicz University, Pozna\'n, 61-614 Pozna\'n, Poland. e-mail: \url{kjerzy@amu.edu.pl}

\medskip
\noindent
Łukasz Pańkowski, Faculty of Mathematics and Computer Science, Adam Mickiewicz University, Pozna\'n, 61-614 Pozna\'n, Poland. e-mail: \url{lpan@amu.edu.pl}

\medskip
\noindent
Maciej Radziejewski, Faculty of Mathematics and Computer Science, Adam Mickiewicz University, Pozna\'n, 61-614 Pozna\'n, Poland. e-mail: \url{maciejr@amu.edu.pl}

\end{document}